\documentclass[12pt]{article}

\usepackage[utf8]{inputenc}

\usepackage[american]{babel}

\usepackage{anysize}
\marginsize{2.3cm}{2.3cm}{2.3cm}{2.3cm}
\usepackage{amsmath}
\usepackage{amssymb}
\usepackage{latexsym}

\usepackage{amsthm}

\usepackage{color}

\usepackage{ifpdf}

\ifpdf
\usepackage[pdftex]{graphicx}
\else
\usepackage{graphicx}
\fi

\title{}
\date{}

\newtheorem{theorem}{Theorem}
\newtheorem{lemma}[theorem]{Lemma}

\newtheorem{stat*}{Statement}

\theoremstyle{remark}
\newtheorem{remark}{Remark}

\begin{document}

\ifpdf
\DeclareGraphicsExtensions{.pdf, .jpg, .tif}
\else
\DeclareGraphicsExtensions{.eps, .jpg}
\fi

\title {On the number of points in a lattice polytope}
\author{Arseniy Akopyan\thanks{University of Texas at Brownsville (USA), Institute of Systems Analysis (Russia).  Research supported by RFBR grants 08-01-00565-a and 10-01-00096-a} \and Makoto Tagami\thanks{Mathematical Institute, Tohoku University}}

\maketitle

\begin{abstract}
	In this article we will show that for every natural $d$ and $n>1$ there exists a natural number~$t$ such that for every $d$-dimensional simplicial complex~$\mathcal{T}$ with vertices in $\mathbb{Z}^d$, the number of lattice points in the $t^{\mathrm{th}}$ dilate of $\mathcal{T}$ is exactly $\chi(\mathcal{T})$ modulo~$n$, where $\chi(\mathcal{T})$ is the Euler characteristic of $\mathcal{T}$.
\end{abstract}

\section{Introduction}

This problem was given to one of the authors by Rom Pinchasi. He noticed that if we scale a segment with vertices in a lattice in two times, then the number of lattice points in the scaled segment will be odd.
For polygons with vertices in a two-dimensional lattice, the same fact follows from Pick's formula except that this polygon must be scaled in four times. We will show that the following theorem holds:

\begin{theorem}
	\label{thm:maintheorem}
	For any natural numbers $d$ and $n>1$ there exists a natural number~$t$ such that if $\mathcal{T}$ is any simplicial complex  in $\mathbb{R}^d$ with vertices in the integer lattice $\mathbb{Z}^d$ then the number of lattice points in the complex $t \mathcal{T}$ is equivalent to $\chi(\mathcal{T})$ modulo~$n$.
\end{theorem}

 Here $\chi(\mathcal{T})$ is the Euler characteristic of the complex $\mathcal{T}$ and $t \mathcal{T}$ denotes the image of~$\mathcal{T}$ under similarity with the center at the origin and ratio equal to $t$.

The proof is based on Stanley's theorem on the coefficients of Ehrhart polynomials \cite{stanley1980decompositions}. 
Let us recall definition of Ehrhart polynomial \cite{ehrhart1962polyedres}. A polytope is called a \textit{lattice polytope} if all the vertices lie on $\mathbb{Z}^d$. For any $d$-dimensional lattice polytope $\mathcal{P}$ in $\mathbb{R}^d$, there exists a polynomial
\begin{equation}
	L(\mathcal{P}, t)=a_dt^d + a_{d-1}t^{d-1} +\cdots + a_0,
\end{equation}
such that the number of lattice points in the polytope $t \mathcal{P}$ is equal to $L(\mathcal{P}, t)$.
It is possible to prove that $a_0$ is the Euler characteristic of
$\mathcal{P}$ (that is one for convex polytopes) and $a_d$ is the volume of $\mathcal{P}$.
Further important properties of Ehrhart polynomial and its connection with number theory, combinatorics and discrete geometry could be found in \cite{beck2006computing}.

\section{Proof}

 First we prove the following lemma. Here $\left[\cdot \right]$ is the floor function, that is, $\left[x \right]$ denotes the largest integer number not greater than $x$.

\begin{lemma}
	\label{lem:convex lemma}
	Let $\mathcal{P}$ be a convex polytope in $\mathbb{R}^d$ with vertices in the integer lattice $\mathbb{Z}^d$, $p$ be any prime number and $l=\left[\log_p{d}\right]$.
	Then for any natural number $k>l$, the number of lattice points in the convex polytope $p^k \mathcal{P}$ is exactly one modulo ${p^{k-l}}$.
\end{lemma}

% \begin{lemma}
% 	\label{lem:maintheorem}
% 	Let $\mathcal{P}$ is an simplicial complex in $\mathbb{R}^d$ with vertices in a integer lattice $\mathbb{Z}^d$.
% 	Suppose $p$ is prime number and $k$ is such natural number that $p^k>d$.
% 	Then simplicial complex $p^k \mathcal{P}$ contains $\chi(\mathcal{P}) \pmod{p}$ points from $\mathbb{Z}^d$.
% \end{lemma}

\begin{proof}
	From Stanley's nonnegativity theorem (more precisely Lemma 3.14 in \cite{beck2006computing}) it follows that in this case the number of lattice points in the convex polytope~$t \mathcal{P}$ equals exactly:
\begin{equation}
	\label{eq:stanley represent}
		{t+d \choose d}+h_1 {t+d-1 \choose d}+\dots +h_{d-1}{t+1 \choose d}+h_d{t\choose d},
\end{equation}
where $h_1$, $h_2$, \dots, $h_d$ are nonnegative integer numbers.
	
	Suppose $t=p^k$ and $m\le d \le p^{l+1}-1$.
	If $\alpha$ is the maximal power of $p$ which divides~$m$ then $(m+p^k)/p^\alpha \equiv m/p^\alpha \pmod {p^{k-l}}$.
	Using this fact it is easy to show that ${t+d \choose d}\equiv 1 \pmod {p^{k-l}}$.
	Also from Kummer's theorem (see \cite{graham1988concrete}, exercise 5.36) it follows that for any $i=1, 2, \dots, d$  we have ${t+d -i\choose d}\equiv 0 \pmod {p^{k-l}}$.
	So as we can see, the number of lattice points equals exactly one modulo ${p^{k-l}}$.
	\end{proof}

	\begin{remark}
		\label{re:apk}
		It is easy to see that the statement of Lemma \ref{lem:convex lemma} holds for dilation factor $ap^k$, $a \in \mathbb{N}$. For proof it is sufficient to apply the Lemma to the polytope $a \mathcal{P}$.
	\end{remark}

	% Now it is easy to prove the general case.

\begin{proof}[Proof of Theorem \ref{thm:maintheorem}]
	Consider the prime factorization of $n$:
	$$
	n=p_1^{\alpha_1}p_2^{\alpha_2}p_3^{\alpha_3}\dots p_s^{\alpha_s}.
	$$
	
	Let $\beta_i=\alpha_i+\left[\log_{p_i}{d}\right]$.
	Define
	$
	t=p_1^{\beta_1}p_2^{\beta_2}p_3^{\beta_3}\dots p_s^{\beta_s}.
	$
	Suppose $\Delta$ is a simplex. By Lemma \ref{lem:convex lemma} we have that the number of lattice points in $t \Delta$ equals $1$ modulo $p_i^{\alpha_i}$ for any $i=1,2,\dots,s$. From the Chinese remainder theorem, it follows that this number is equivalent to $1$ modulo~$n$.
	
		We know that the Euler characteristic of every simplex (with its interior) equals~$1$ and the Euler characteristic is an additive function on simplicial complexes. Since the number of lattice points modulo $n$ is also an additive function, we obtain that the number of lattice points is equivalent to exactly $\chi(\mathcal{T}) \pmod{n}$.
\end{proof}

\begin{remark}
	\label{re:minimal} As noted by the anonymous referee the statement of Theorem \ref{thm:maintheorem} is kind of obvious for $t=nd!$.
	It is well-known that for any $d$-dimensional lattice polytope, 
all the coefficients of the Ehrhart polynomial are rational numbers and all the denominators except for the constant term $1$ are divisors of $d!$.
	In other words, the polynomial is of the form
	     $$L(\mathcal{P}, t)=1+t\cdot p(t)/d!$$
	where the polynomial $p(t)$ has integer coefficients. So if $t=n\cdot d!$ then
$$
L(\mathcal{P}, nd!)=1+n\cdot p(n\cdot d!)
$$
	which is $1$ modulo $n$.
	
	Let us show that the number $t$ obtained in the proof of Theorem \ref{thm:maintheorem} is the minimal natural number which satisfies the condition of the Theorem.  
	
	Suppose $t$ is not divisible by $p_i^{\beta_i}$ for some $i$. Let $d'=p_{i}^{\left[\log_{p_i}{d}\right]}$ and $\Delta$ be a $d'$-dimensional simplex with vertices $(0, 0, \dots, 0)$, $(1, 0, \dots, 0)$, \dots, $(0, 0, \dots, 1)$.
	Then the number of lattice points in the simplex $x \Delta$ is equal to $x+d' \choose d'$ (see \cite{beck2006computing} section 2.3).
	It is easy to see that one can choose $x$ such that $t\cdot x\equiv p_i^{\beta_i-1} \pmod {p_i^{\beta_i}}$. Note that if $a \equiv b \pmod {p_i^{\beta_i}}$, then 

$$
	{a +k \choose k} \equiv {b+k \choose k} \pmod  {p_i^{\alpha_i}},
	\text{ for all } k<p_i^{\beta_i-\alpha_i}=d'.
$$

	Since ${p_i^{\beta_i-1}+d'-1 \choose d'-1}\equiv 1 \pmod {p_i^{\alpha_i}}$, we have
		\begin{multline}
	L(x \Delta, t)={xt +d' \choose d'}\equiv
	{p_i^{\beta_i-1}+d' \choose d'}	=
	\\
	={p_i^{\beta_i-1}+d' -1 \choose d' -1}\cdot \frac {p_i^{\beta_i-1}+d'}{d'} \equiv p_i^{\alpha_i-1}+1 \pmod {p_i^{\alpha_i}}.
	\end{multline}

\end{remark}

{\bf Acknowledgment.} We wish to thank Rom Pinchasi and Oleg Musin for helpful discussions and remarks. We are grateful to the anonymous referee of the journal Integers for his valuable comments.

% \nocite{ehrhart1962polyedres}
% \nocite{stanley1980decompositions}
% \bibliographystyle{plain}
% \bibliography{../BibTeX/Bibliography}{}

\end{document}